\documentclass[11pt,reqno]{amsart}
\usepackage[T1]{fontenc}
\usepackage{lmodern}
\usepackage{amsmath,amssymb,mathtools}
\usepackage[letterpaper,textwidth=6.05in,textheight=8.8in,centering]{geometry}
\usepackage{microtype}
\usepackage{needspace}
\usepackage{tikz-cd}
\usepackage{hyperref}

\newtheorem{theorem}{Theorem}[section]
\newtheorem{proposition}[theorem]{Proposition}
\newtheorem{lemma}[theorem]{Lemma}
\newtheorem{corollary}[theorem]{Corollary}
\theoremstyle{definition}
\newtheorem{example}[theorem]{Example}
\theoremstyle{remark}
\newtheorem{remark}[theorem]{Remark}
\numberwithin{equation}{section}
\title[Degree-two Chow varieties]{Poincar\'e series of degree-two Chow varieties}

\author[Y. Chen]{Youming Chen}
\address{School of Mathematical Science, Chongqing University of Technology, Chongqing 400054, P.R. China}
\email{youmingchen@cqut.edu.cn}

\author[W. Hu]{Wenchuan Hu}
\address{School of Mathematics, Sichuan University, Chengdu 610064, P.R. China}
\email{huwenchuan@gmail.com}

\author[D. Ma]{Deyuan Ma}
\address{School of Mathematics, Sichuan University, Chengdu 610064, P.R. China}
\email{2024222010027@stu.scu.edu.cn}

\date{\today}
\begin{document}
\begin{abstract}
Let $C_{p,2}(\mathbb{P}^{n})$ denote the Chow variety parametrizing effective algebraic $p$-cycles of degree two on $\mathbb{P}^{n}$.
We show that its rational cohomology is pure Tate and concentrated in even degrees.
An anti-invariant Chern class identity yields a presentation of the cohomology ring as a fiber product,
which in turn determines the cohomology over $\mathbb{Z}[1/2]$.
We identify the ordinary Poincar\'e polynomial with the known virtual formula,
determine the optimal ranges over which the coefficients stabilize,
and compute the nonrational Poincar\'e series of the stable space $D(2)$.
For the Lawson–Michelsohn map $D(2)\longrightarrow D(\infty)$,
the first nonzero rational homotopy kernel occurs in degree eight and is one-dimensional,
while the integral kernel has rank one.
Finally, we compute the relative homology of the pair $(D(2),BU)$ over $\mathbb{Z}[1/2]$
and show that additional cohomology generators and nonzero rational homotopy kernels arise in unbounded degrees.
\end{abstract}
\maketitle

\section{Introduction}
Let $\mathbb{P}^{n}$ denote complex projective $n$-space and let
$C_{p,d}(\mathbb{P}^{n})$ be the Chow variety of effective algebraic
$p$-cycles of degree $d$, where $0\leq p\leq n$ and $d\geq1$.
These projective varieties need not be smooth or irreducible.
In degree two, the decomposition into sums of linear cycles and
irreducible quadrics gives a particularly explicit description
\cite[\S2, Eq.~(2.1), P.~2366]{Hu10}. To determine ordinary cohomology from this decomposition, one must
analyze the associated connecting homomorphisms.

The main contribution of this paper is an identity for anti-invariant
Chern classes that makes the relevant difference of restriction maps
surjective. It follows that $C_{p,2}(\mathbb{P}^{n})$ has pure Tate
rational cohomology and no odd rational cohomology. The same argument,
carried out over $\mathbb{Z}[1/2]$, gives a presentation of the cohomology
ring and excludes odd primary torsion in integral homology and
cohomology. Rational formality of the finite stage varieties is then an
application of Cirici and Horel's purity theorem
\cite[Thm.~8.1 and Prop.~8.9(1)]{CH20}.

We use Hu's virtual Poincar\'e formula
\cite[\S4.5, Ex.~4.2]{Hu21}. Purity identifies this established
formula with the ordinary Poincar\'e polynomial. We determine the
leading terms of the differences in the two stabilization limits and
prove that the iterated stable series is nonrational.
The geometric decomposition, together with the natural homotopy stability theorem
used to identify this series with the cohomology of $D(2)$, is due to \cite{Hu10}.

Lawson and Michelsohn introduced the degree filtration
$BU=D(1)\subset D(2)\subset\cdots\subset D(\infty)$ and asked whether
the maps to $D(\infty)$ induce injections on homotopy groups
\cite[\S7, Ques.~I, P.~263]{LM}. Hu answered this negatively for
degree two \cite[Thm.~1.2, P.~2366]{Hu10}; his argument already
uses homotopy groups only through degree nine
\cite[Rem.~2.7, P.~2369]{Hu10}. We determine the first failure after
rationalization: it occurs in degree eight, where the kernel is one
dimensional and $\pi_8(D(2))\otimes\mathbb{Q}\cong\mathbb{Q}^{2}$.
The corresponding integral kernel has rank one, and the conclusion
already holds at $C_{4,2}(\mathbb{P}^{8})$.

For comparison, Hu computed
$(b_{2s}-b_{2s+1})(C_{4,2}(\mathbb{P}^{n}))$, for $0\leq s\leq4$
and $n\geq9$, obtaining $(1,1,2,3,6)$
\cite[Prop.~2.6, P.~2369]{Hu10}.
The odd vanishing assertion in his Corollary~2.3 is conditional on the
injectivity hypothesis used in that paper's contradiction argument.
Our restriction map calculation determines both parities in every
degree and supplies the cohomological justification for using the
known virtual formula as an ordinary Poincar\'e polynomial.

We also compute the relative homology of $(D(2),BU)$ after inverting
$2$. This addresses the degree-two case of Lawson and Michelsohn's
Question~IV \cite[\S7, P.~263]{LM}, without determining possible
$2$-primary torsion. Finally, a growth comparison shows that the stable
cohomology is not finitely generated as an algebra over the classical
Chern class subring and that rational homotopy kernels occur in
infinitely many degrees.

Section~\ref{sec:main} states the main results.
Sections~\ref{sec:geometry} and~\ref{sec:purity} establish the
restriction map calculation and its consequences.
Section~\ref{sec:series} treats the Poincar\'e series, and
Section~\ref{sec:homotopy} gives the homotopy applications and relative
homology calculation. Examples appear in Section~\ref{sec:examples}.

\subsection*{Declaration concerning AI assistance}

ChatGPT was used as an auxiliary tool in manuscript preparation.
The authors are responsible for the final text and mathematical arguments.

\section{Main results}\label{sec:main}
All varieties are defined over $\mathbb{C}$ and have their reduced structures; irreducibility is not assumed.
All homology and cohomology are taken in the complex analytic topology.
For $n\geq0$ and $0\leq p\leq n$, we use the notation
\begin{equation*}
 X_{p,n}=C_{p,2}(\mathbb{P}^{n}),\quad
 H^m(Z)=H^m(Z;\mathbb{Q}),\quad b_m(Z)=\dim_{\mathbb{Q}}H^m(Z).
\end{equation*}
Unspecified homology and cohomology coefficients are rational.
For a commutative coefficient ring $K$, let
$H_K^m(Z)=H^m(Z;K)$ and $H_K^m(Z,A)=H^m(Z,A;K)$.
Throughout, $\Lambda$ denotes the localization $\mathbb{Z}[1/2]$.
Tensor products are over $\mathbb{Z}$ unless a base ring is specified.
For a space with finite Betti numbers in each degree, its Poincar\'e
series is $P_Z(t)=\sum_{m\geq0}b_m(Z)t^m$. For a pair $(Z,A)$,
the analogous series is defined using relative cohomology.
If $Z$ has finite CW type, its Euler characteristic is $\chi(Z)=P_Z(-1)$.
Negative degree groups are zero; empty sums and products are $0$ and
$1$, respectively.

Projectivization $\mathbb{P}(E)$ parametrizes lines in a vector bundle
$E$, with tautological subbundle $\mathcal{O}_{\mathbb{P}(E)}(-1)$.
For a bundle $E$ on $Z$, the Chern classes $c_j(E)$ belong to
$H^{2j}(Z;\mathbb{Z})$; the same notation denotes their images
after a change of coefficients. We adopt the conventions $c_0(E)=1$ and $c_j(E)=0$ outside
$0\leq j\leq\operatorname{rank}E$. We write
$c(E)=\sum_j c_j(E)$, use $E^\vee$ for a dual bundle or vector space,
$\operatorname{SP}^{d}(Z)=Z^d/\mathfrak{S}_d$ for the symmetric product,
where $\mathfrak{S}_d$ is the symmetric group, and
$\operatorname{Sym}^{d}(E)$ for the symmetric power of a bundle.
The Tate structure $\mathbb{Q}(-s)$ has weight $2s$ and Hodge type
$(s,s)$.

\Needspace{7\baselineskip}
\begin{theorem}\label{thm:purity}
For $n\geq0$, $0\leq p\leq n$, and $s\geq0$,
\begin{equation*}
 H^{2s+1}(X_{p,n})=0,\quad
 H^{2s}(X_{p,n})\cong\mathbb{Q}(-s)^{\oplus b_{2s}(X_{p,n})}
\end{equation*}
as rational mixed Hodge structures.
\end{theorem}

Let $q$ be an indeterminate. Define the polynomials
\begin{equation*}
 \Pi_r(q)=\sum_{a=0}^{r}q^a,\quad
 G_{r,N}(q)=\begin{bmatrix}N\\r\end{bmatrix}_{q}
 =\prod_{j=1}^{r}\frac{1-q^{N-r+j}}{1-q^j}.
\end{equation*}
Here $r\geq0$ in the first formula and $0\leq r\leq N$ in the
second. Thus $G_{r,N}$ is a Gaussian binomial polynomial and
$G_{0,N}=1$. Let $m_k=\binom{k+2}{2}-1$ for $k\geq1$. Whenever
$p<n$, we write $k=p+1$ and $N=n+1$. For a formal series $F(q)$,
$[q^s]F(q)$ denotes the coefficient of $q^s$.
Let $E(Z;u,v)$ be the compactly supported Hodge--Deligne polynomial
and define $P_Z^{\mathrm{vir}}(t)=E(Z;t,t)$, with the signed convention
specified in Section~\ref{sec:series}.

\begin{theorem}\label{thm:finite}
For $0\leq p<n$, let
\begin{equation}\label{eq:finite}
\begin{split}
 F_{p,n}(q)={}&\frac{G_{k,N}(q)^2+G_{k,N}(q^2)}{2}\\
 &+G_{k+1,N}(q)
 \left[\Pi_{m_k}(q)-\frac{\Pi_k(q)^2+\Pi_k(q^2)}{2}\right].
\end{split}
\end{equation}
Then
\begin{equation*}
 E(X_{p,n};u,v)=F_{p,n}(uv),\quad
 P_{X_{p,n}}(t)=P_{X_{p,n}}^{\mathrm{vir}}(t)=F_{p,n}(t^2).
\end{equation*}
For $p=n$, these identities hold with $F_{n,n}=1$.
The virtual polynomial in \eqref{eq:finite} is the formula of
\cite[\S4.5, Ex.~4.2]{Hu21}; the assertion here identifies it with
ordinary cohomology.
\end{theorem}

All formal series limits below are coefficientwise: each coefficient
is eventually constant. The notation $O(q^a)$ denotes an element of
$q^a\mathbb{Q}[[q]]$. Define
\begin{equation*}
 U_r(q)=\prod_{j=1}^{r}(1-q^j)^{-1},\quad
 U(q)=\prod_{j\geq1}(1-q^j)^{-1},\quad U_0(q)=1.
\end{equation*}
The infinite product defines an element of $\mathbb{Z}[[q]]$.

\begin{theorem}\label{thm:stable}
For fixed $p\geq0$ and $k=p+1$,
\begin{equation}\label{eq:fixed}
\begin{split}
 F_p(q):=\lim_{n\longrightarrow\infty}F_{p,n}(q)
 ={}&\frac{U_k(q)^2+U_k(q^2)}{2}\\
 &+U_{k+1}(q)
 \left[\Pi_{m_k}(q)-\frac{\Pi_k(q)^2+\Pi_k(q^2)}{2}\right]
 \in\mathbb{Q}(q)\cap\mathbb{Z}[[q]].
\end{split}
\end{equation}
Moreover,
\begin{equation}\label{eq:stable}
 F_\infty(q):=\lim_{p\longrightarrow\infty}F_p(q)
 =\frac{U(q)^2+U(q^2)}{2}
 -\frac{q^2U(q)}{(1-q)(1-q^2)}.
\end{equation}
Neither $F_\infty(q)$ nor $F_\infty(t^2)$ is rational over $\mathbb{C}$.
\end{theorem}

Following Lawson and Michelsohn \cite[\S7, Eq.~ (7.1)--(7.2)]{LM},
define
\begin{equation*}
 D(d)=\operatorname*{colim}_{a,b\longrightarrow\infty}
 C_{a,d}(\mathbb{P}^{a+b}),\quad
 D(\infty)=\operatorname*{colim}_{d}D(d).
\end{equation*}
The transition maps in $a$ are algebraic suspension, and those in $b$
are ambient linear inclusions. The degree maps add a fixed linear
cycle. Choose compatible linear flags and basepoints given by $d$
times the distinguished linear $a$-plane. We use the weak topology on
these limits and CW models in homotopy arguments.
The classical identifications are
\begin{equation*}
 D(1)=BU,\quad
 K_{\mathrm{ev}}:=D(\infty)\simeq\prod_{j\geq1}'K(\mathbb{Z},2j),
\end{equation*}
where $BU$ is the classifying space of the stable unitary group,
the prime denotes the weak product (the union of finite products),
and $K(\mathbb{Z},2j)$ is an Eilenberg--MacLane space.
Denote the filtration maps and the stabilization map by
\begin{equation*}
 i:BU\longrightarrow D(2),\quad
 J:D(2)\longrightarrow K_{\mathrm{ev}},\quad
 s_{p,n}:X_{p,n}\longrightarrow D(2),\quad \vartheta=J\circ i.
\end{equation*}
For $j\geq1$, let $u_j\in H^{2j}(K_{\mathrm{ev}};\mathbb{Z})$ be
the pullback of the universal cohomology class of $K(\mathbb{Z},2j)$
under the projection determined by the chosen weak product identification,
and let $c_j\in H^{2j}(BU;\mathbb{Z})$ denote the universal
$j$-th Chern class. With the identifications of
\cite[Thm.~3.6 and Eq.~ (7.2)]{LM}, one has
$\vartheta^*u_j=c_j$.
For a based map $f$ and $m\geq2$, denote its induced homomorphism
on rational homotopy groups by
$f_{m,\mathbb{Q}}=\pi_m(f)\otimes\mathbb{Q}$. The spaces occurring below are simply connected.

\begin{theorem}\label{thm:first}
The ordinary Poincar\'e series of $D(2)$ is $F_\infty(t^2)$.
The map $J_{m,\mathbb{Q}}$ is an isomorphism for $2\leq m\leq7$, and
\begin{equation*}
 \pi_8(D(2))\otimes\mathbb{Q}\cong\mathbb{Q}^{2},\quad
 \ker J_{8,\mathbb{Q}}\cong\mathbb{Q}.
\end{equation*}
The integral kernel of $\pi_8(J):\pi_8(D(2))\longrightarrow\mathbb{Z}$
has rank one and contains an element of infinite order.
If $p\geq4$ and $n-p\geq4$, the composite $J\circ s_{p,n}$ has the
same properties on $\pi_8$. In particular, $C_{4,2}(\mathbb{P}^{8})$
is a finite stage example.
\end{theorem}

For a nonnegatively graded algebra $E$, write
$E^{>0}=\bigoplus_{m>0}E^m$ and let $(E^{>0})^2$ be the span of
products of two positive degree elements.
A graded algebra over a field is connected when its degree zero part
is that field.

\begin{theorem}\label{thm:unbounded}
Let $H=H^*(D(2))$ and $A_{\mathrm{cl}}=J^*H^*(K_{\mathrm{ev}})$.
The algebra $H$ is not finitely generated as an $A_{\mathrm{cl}}$-algebra.
Equivalently, the graded vector space
\begin{equation*}
 Q_{\mathrm{cl}}(H)=
 H^{>0}\big/\bigl((H^{>0})^2+A_{\mathrm{cl}}^{>0}\bigr)
\end{equation*}
is infinite dimensional. It is nonzero in infinitely many even degrees,
all at least eight. Furthermore, $\ker J_{m,\mathbb{Q}}\neq0$ for
infinitely many $m\geq8$. In each such degree the integral kernel
contains an element of infinite order. These homotopy conclusions
hold for $J\circ s_{p,n}$ whenever
$m\leq\min\{2p+1,2(n-p)\}$.
\end{theorem}

\section{Geometry and cohomological preliminaries}\label{sec:geometry}

In this section and Section~\ref{sec:purity}, assume $1\leq p<n$
unless otherwise stated. The classification of degree two cycles used
below is \cite[\S2, Eq.~(2.1), P.~2366]{Hu10}.
Irreducible quadrics are understood to include singular irreducible
quadrics; their ranks are at least three. This is the irreducible locus
of the complement appearing in that decomposition.

Let
\begin{equation*}
 V=\mathbb{C}^{n+1},\quad k=p+1,\quad
 G=\operatorname{Gr}(k,V),\quad T=\operatorname{Gr}(k+1,V),
\end{equation*}
where $\operatorname{Gr}(r,V)$ parametrizes $r$-dimensional subspaces.
Let $S$ be the rank-$(k+1)$ tautological subbundle on $T$, and define
\begin{equation*}
 R=\operatorname{SP}^{2}(G),\quad
 Y=\mathbb{P}(\operatorname{Sym}^{2}(S^\vee)),\quad
 D=\mathbb{P}(S^\vee),\quad
 \widetilde B=D\times_TD,\quad B=\widetilde B/\mathfrak{S}_2.
\end{equation*}
A point of $D$ is a flag $E\subset W\subset V$ with
$\dim E=k$ and $\dim W=k+1$. Write $\pi:D\longrightarrow T$ and
$\operatorname{pr}_1,\operatorname{pr}_2:\widetilde B\longrightarrow D$
for the projections. The involution on $\widetilde B$ exchanges its
two flags. Multiplication of linear forms defines
\begin{equation*}
 j:B\longrightarrow Y,\quad
 (W,\{\ker\ell_1,\ker\ell_2\})\longmapsto(W,[\ell_1\ell_2]).
\end{equation*}
This is a homeomorphism onto the closed rank at most two locus.
The map $f:B\longrightarrow R$ forgets $W$.
For $X=X_{p,n}$, addition of cycles gives $r:R\longrightarrow X$,
and the universal family of quadrics gives $g:Y\longrightarrow X$.
The map $r$ identifies $R$ homeomorphically with the decomposable locus,
and $g\circ j=r\circ f$.
The spaces $Y,B$ are the spaces denoted $A_{p,n},B_{p,n}$ in
\cite[\S2, P.~2368]{Hu10}. Only the stated topological identifications
are needed.

\begin{proposition}
The square
\begin{equation*}
\begin{tikzcd}[column sep=large,row sep=large]
 B\arrow[r,"j"]\arrow[d,"f"']&Y\arrow[d,"g"]\\
 R\arrow[r,"r"']&X
\end{tikzcd}
\end{equation*}
is a pushout of compact Hausdorff spaces. For every commutative
coefficient ring $K$, it induces
$H_K^m(X,R)\cong H_K^m(Y,B)$ and the exact sequence
\begin{equation}\label{eq:mv}
 \cdots\longrightarrow H_K^m(X)
 \xrightarrow{(g^*,r^*)}H_K^m(Y)\oplus H_K^m(R)
 \xrightarrow{j^*-f^*}H_K^m(B)
 \longrightarrow H_K^{m+1}(X)\longrightarrow\cdots.
\end{equation}
\end{proposition}
\begin{proof}
An irreducible cycle has a unique preimage in $Y$,  determined by its linear span.
A decomposable cycle has a unique preimage in $R$,
and all its preimages in $Y$ lie in $B$ and are identified with that point by $f$.
For a double plane $2[\mathbb{P}(E)]$, the choices of $W$ form $\mathbb{P}(V/E)$.
For two distinct $k$-dimensional subspaces $E_1,E_2\subset V$,
such a $W$ exists precisely when $\dim(E_1+E_2)=k+1$; in that case $W=E_1+E_2$.
Consequently $Y\cup_B R\longrightarrow X$ is a continuous bijection
from a compact space to a Hausdorff space, hence a homeomorphism.

Compatible  semialgebraic triangulation
\cite[Thm.~9.2.1, P. ~217]{BCR} makes $(Y,j(B))$ and $(X,r(R))$
finite CW pairs, since the distinguished subsets are closed.
The quotient identification $Y/B\cong X/R$ gives the relative
comparison \cite[Prop.~2.22, P.~124]{Hatcher}.
Since $j$ is a cofibration, the pushout is represented up to homotopy
by the double mapping cylinder. Its cohomological Mayer--Vietoris
sequence is \eqref{eq:mv}; see \cite[\S2.2, P.~152; \S3.1]{Hatcher}.
\end{proof}

We recall the cohomological results used in the sequel. If a finite group $\Gamma$ acts algebraically on a
projective variety $Z$ and $|\Gamma|$ is invertible in $K$, the quotient
map induces the finite quotient comparison
\begin{equation}\label{eq:quotient}
 H_K^*(Z/\Gamma)\xrightarrow{\ \cong\ }H_K^*(Z)^\Gamma.
\end{equation}
This is a consequence of proper base change for finite fibers \cite[Tag~09V5, Lem.~20.18.1]{Stacks}
and together with the exactness of the invariant functor.
No freeness of the action is required. For an equivariant homomorphism $\varphi:M\longrightarrow N$
of $K[\Gamma]$-modules, exactness of invariants also gives
$\varphi(M^\Gamma)=(\operatorname{im}\varphi)^\Gamma$.
Over $\mathbb{Q}$, \eqref{eq:quotient} respects mixed Hodge structures
by algebraic functoriality. We use Deligne's functoriality and
strictness throughout \cite[Thm.~2.3.5]{DeligneII} and \cite[Prop.~8.2.2, 8.2.10 and~8.3.9]{DeligneIII}.

The integral cohomology of a complex Grassmannian is free, concentrated
in even degrees, and generated by tautological Chern classes
\cite[\S9.4, P.~146--148; \S10.6, Prop.~8--9, P.~178--180]{Fulton}.
We also use the integral projective bundle formula with the line
convention of \cite[App.~B.4, Eq.~(32)--(33), P.~223]{Fulton}:
for a rank-$e$ bundle $E$ on a smooth projective variety $Z$,
$H^*(\mathbb{P}(E);\mathbb{Z})$ is free over $H^*(Z;\mathbb{Z})$
with basis $1,\zeta,\ldots,\zeta^{e-1}$, where
$\zeta=c_1(\mathcal{O}_{\mathbb{P}(E)}(1))$. The relation is
\begin{equation}\label{eq:pb}
 \zeta^e+c_1(E)\zeta^{e-1}+\cdots+c_e(E)=0,
\end{equation}
with pullbacks from the base suppressed.

For the remainder of Sections~\ref{sec:geometry}--\ref{sec:purity},
$K$ denotes either $\mathbb{Q}$ or $\Lambda$ unless specified otherwise.
The preceding facts imply that $G,T,Y,\widetilde B,R,B$ have
cohomology over $K$ concentrated in even degrees. Over $\Lambda$ it
is finite free in each degree; for $R,B$,
the invariants are direct summands of finite free modules over the principal ideal domain $\Lambda$.

Let $q_B:\widetilde B\longrightarrow B$ be the quotient,
$\delta:D\longrightarrow\widetilde B$ the diagonal, and
$\iota=q_B\circ\delta$. Define
\begin{equation*}
 \xi=c_1(\mathcal{O}_D(1)),\quad
 x=\operatorname{pr}_1^*\xi,\quad y=\operatorname{pr}_2^*\xi.
\end{equation*}
Let $E_1,E_2\subset S$ be the two universal hyperplane subbundles on
$\widetilde B$, with pullbacks understood. We use the notation
\begin{equation*}
 a_j=c_j(E_1),\quad b_j=c_j(E_2),\quad e_j=c_j(S),\quad
 c=e_1,\quad d_j=a_j-b_j.
\end{equation*}
Here $a_0=b_0=e_0=1$, and $a_j=b_j=0$ for $j>k$.
The tautological sequence on $D$ identifies
$x=c_1(S/E_1)$ and $y=c_1(S/E_2)$. Whitney's formula gives
\begin{equation}\label{eq:whitney}
 e_j=a_j+xa_{j-1}=b_j+yb_{j-1},\quad
 x=c-a_1,\quad y=c-b_1.
\end{equation}
Write $K_T=H_K^*(T)$ and
$\Theta(z)=\sum_{j=0}^{k+1}(-1)^j e_jz^{k+1-j}$.
The projective bundle formula gives
\begin{equation}\label{eq:diagonalrings}
 H_K^*(D)=K_T[\xi]/(\Theta(\xi)),\quad
 H_K^*(\widetilde B)=K_T[x,y]/(\Theta(x),\Theta(y)).
\end{equation}
The signs come from $c_j(S^\vee)=(-1)^je_j$.

\begin{lemma}\label{lem:diagonal}
The map $\delta^*:H_K^*(\widetilde B)\longrightarrow H_K^*(D)$ is
surjective and has kernel $(x-y)H_K^*(\widetilde B)$.
\end{lemma}
\begin{proof}
In \eqref{eq:diagonalrings}, $\delta^*$ fixes $K_T$ and sends both
$x$ and $y$ to $\xi$. The quotient by $(x-y)$ is exactly the first ring.
\end{proof}

\section{Restriction surjectivity and cohomology}\label{sec:purity}
The map
$\widetilde f:\widetilde B\longrightarrow G\times G$ forgets $W$.
For the coefficient ring $K$ fixed above, define the graded algebras
\begin{equation*}
 C=H_K^*(\widetilde B),\quad
 A=\operatorname{im}\bigl(\widetilde f^*:H_K^*(G\times G)
 \longrightarrow C\bigr).
\end{equation*}
Let $\sigma$ denote the involution induced by interchange of the two factors.
For a $\sigma$-stable module $M$, we write $M^+=\{v\in M:\sigma v=v\}$ and $M^-=\{v\in M:\sigma v=-v\}$.
The integral Chern class generation recalled above gives
$A=K[a_1,\ldots,a_k,b_1,\ldots,b_k]\subset C$,
subject to its actual relations. We do not regard these classes as algebraically independent.

\begin{lemma}\label{lem:anti}
One has $C=A[c]$ and $C^-=A^-$.
\end{lemma}
\begin{proof}
The ring $C$ is generated by $e_1,\ldots,e_{k+1},x,y$.
Equation~\eqref{eq:whitney} expresses all these generators in $A[c]$.
In the polynomial ring
$K[A_1,\ldots,A_k,B_1,\ldots,B_k]$, consider the linear change of
variables $s_j=A_j+B_j$, $t_j=A_j-B_j$. It is invertible because
$2$ is a unit in $K$. An anti-invariant monomial has odd total $t$-degree
and is of the form $t_j$ times an invariant monomial. Taking anti-invariants
under the equivariant surjection onto $A$ yields
\begin{equation*}
 A^-=\sum_{j=1}^{k}A^+d_j.
\end{equation*}
Comparing the two expressions for $e_{j+1}$ gives
\begin{equation*}
 cd_j=a_1a_j-b_1b_j-d_{j+1}\quad(1\leq j\leq k),\quad d_{k+1}=0.
\end{equation*}
Thus $cA^-\subset A^-$.
Since $c$ is invariant, any expression $\omega=\sum_m\alpha_mc^m\in C^-$, with
$\alpha_m\in A$, satisfies
\begin{equation*}
 \omega=\sum_m\frac{\alpha_m-\sigma(\alpha_m)}{2}c^m\in A^-.
\end{equation*}
The reverse inclusion is immediate. All identities hold in the image algebra $A$, without any assumption
of algebraic independence.
\end{proof}

\Needspace{6\baselineskip}
\begin{lemma}\label{lem:kernel}
One has
$\ker(\iota^*:H_K^*(B)\longrightarrow H_K^*(D))
\subset\operatorname{im}(f^*:H_K^*(R)\longrightarrow H_K^*(B))$.
\end{lemma}
\begin{proof}
The quotient comparison \eqref{eq:quotient} identifies $H_K^*(B)$
with $C^+$ and $\operatorname{im}f^*$ with $A^+$, by equivariance
and exactness of invariants.
For $\alpha\in C^+$ vanishing on the diagonal,
Lemma~\ref{lem:diagonal} gives $\alpha=(x-y)\beta$. Hence
\begin{equation*}
 \alpha=(x-y)\frac{\beta-\sigma(\beta)}{2}.
\end{equation*}
Both factors belong to $A^-$ by Lemma~\ref{lem:anti}.
Since $x-y=b_1-a_1$,  their product lies in $A^+$.
No cancellation by $x-y$ is used; this class may be a zero divisor.
\end{proof}

\begin{proposition}\label{prop:surjectivity}
For every $m\geq0$, the map
$j^*-f^*:H_K^m(Y)\oplus H_K^m(R)\longrightarrow H_K^m(B)$ is
surjective.
\end{proposition}
\begin{proof}
Let $h=c_1(\mathcal{O}_Y(1))$ and $\widetilde j=j\circ q_B$.
Multiplication of the tautological lines gives
\begin{equation*}
 \widetilde j^{\,*}\mathcal{O}_Y(-1)
 \cong\operatorname{pr}_1^*\mathcal{O}_D(-1)
 \otimes\operatorname{pr}_2^*\mathcal{O}_D(-1).
\end{equation*}
Consequently
\begin{equation}\label{eq:double}
 \widetilde j^{\,*}h=x+y,\quad (j\circ\iota)^*h=2\xi.
\end{equation}
The map $j\circ\iota$ lies over the identity of $T$. The image of
$(j\circ\iota)^*$ therefore contains $K_T$ and $\xi$, since $2$
is invertible. These classes generate $H_K^*(D)$, so
$(j\circ\iota)^*$ is surjective.
Given $\gamma\in H_K^m(B)$, choose $\eta\in H_K^m(Y)$ with
$(j\circ\iota)^*\eta=\iota^*\gamma$.
Lemma~\ref{lem:kernel} gives $\rho\in H_K^m(R)$ such that
$\gamma-j^*\eta=f^*\rho$. Thus
$\gamma=(j^*-f^*)(\eta,-\rho)$.
\end{proof}

\begin{proof}[Proof of Theorem~\ref{thm:purity}]
Take $K=\mathbb{Q}$. Since the cohomology of $Y,R,B$ is even,
\eqref{eq:mv} and Proposition~\ref{prop:surjectivity} imply
$H^{2s+1}(X)=0$ and an injection
\begin{equation*}
 H^{2s}(X)\longrightarrow H^{2s}(Y)\oplus H^{2s}(R).
\end{equation*}
The target is pure Tate: this follows from the algebraic Schubert
classes, the projective bundle formula, and \eqref{eq:quotient}.
The injection is induced by algebraic pullbacks and is strict for
the Hodge and weight filtrations, so its source is pure of type
$(s,s)$ as well.
For $p=0$, use $X_{0,n}=\operatorname{SP}^{2}(\mathbb{P}^{n})$
and \eqref{eq:quotient}; for $p=n$, the variety is the point
$2[\mathbb{P}^{n}]$.
\end{proof}

\begin{corollary}
For $1\leq p<n$ and $K\in\{\mathbb{Q},\Lambda\}$, pullback gives
\begin{equation}\label{eq:fiberproduct}
 H_K^*(X_{p,n})\cong H_K^*(Y)\times_{H_K^*(B)}H_K^*(R).
\end{equation}
The multiplication is componentwise. In degree $2s$ this is the kernel
in the short exact sequence
\begin{equation}\label{eq:short}
 0\longrightarrow H_K^{2s}(X_{p,n})
 \xrightarrow{(g^*,r^*)}H_K^{2s}(Y)\oplus H_K^{2s}(R)
 \xrightarrow{j^*-f^*}H_K^{2s}(B)\longrightarrow0.
\end{equation}
For $K=\mathbb{Q}$, this is an exact sequence of mixed Hodge structures.
More explicitly, with $K_T=H_K^*(T)$ and $\Theta$ as above,
\begin{equation}\label{eq:presentations}
\begin{aligned}
 H_K^*(Y)&=K_T[h]/(\Psi(h)),&
 \Psi(h)&=\sum_{j=0}^{m_k+1}c_j(\operatorname{Sym}^{2}(S^\vee))h^{m_k+1-j},\\
 H_K^*(R)&=\bigl(H_K^*(G)\otimes_KH_K^*(G)\bigr)^{\mathfrak{S}_2},\\
 H_K^*(B)&=\bigl(K_T[x,y]/(\Theta(x),\Theta(y))\bigr)^{\mathfrak{S}_2}.
\end{aligned}
\end{equation}
The last involution fixes $K_T$ and interchanges $x,y$.
The map $j^*$ fixes $K_T$ and sends $h$ to $x+y$.
The map $f^*$ is computed on invariant tautological Chern polynomials
using $a_j=e_j-xa_{j-1}$, $b_j=e_j-yb_{j-1}$.
\end{corollary}
\begin{proof}
The exact sequence \eqref{eq:mv}, the vanishing of odd cohomology,
and Proposition~\ref{prop:surjectivity} give \eqref{eq:short}.
The injective map $(g^*,r^*)$ is a ring map, and its image is the stated
fiber product. The presentations follow from \eqref{eq:pb},
\eqref{eq:quotient}, and \eqref{eq:diagonalrings}; the restriction
formulas are \eqref{eq:whitney} and \eqref{eq:double}.
In $\Psi$, the bundle is $\operatorname{Sym}^{2}(S^\vee)$ itself,
so no additional alternating sign is inserted.
\end{proof}

\begin{corollary}\label{cor:awaytwo}
For $0\leq p\leq n$ and $s\geq0$,
\begin{equation*}
\begin{gathered}
 H^{2s+1}(X_{p,n};\Lambda)=H_{2s+1}(X_{p,n};\Lambda)=0,\\
 H^{2s}(X_{p,n};\Lambda)\cong\Lambda^{b_{2s}(X_{p,n})},\quad
 H_{2s}(X_{p,n};\Lambda)\cong\Lambda^{b_{2s}(X_{p,n})}.
\end{gathered}
\end{equation*}
Integral homology and cohomology have no odd primary torsion.
For every field $F$ of characteristic different from $2$, odd homology
and cohomology vanish and the even dimensions are $b_{2s}(X_{p,n})$.
For $1\leq p<n$, \eqref{eq:fiberproduct}--\eqref{eq:presentations}
also hold over $F$.
\end{corollary}
\begin{proof}
Take $K=\Lambda$ in the preceding argument. The even cohomology of
$X_{p,n}$ is a submodule of a finite free $\Lambda$-module by
\eqref{eq:short}, hence is finite free; its rank is obtained by
extension to $\mathbb{Q}$. Odd cohomology vanishes.
The boundary cases follow from the quotient comparison and the point
case as above.
Since $X_{p,n}$ has finite CW type, the universal coefficient theorem
over the principal ideal domain $\Lambda$ applies to finitely
generated homology groups. Torsion in $H_m(X_{p,n};\Lambda)$ would
contribute a nonzero torsion submodule
$\operatorname{Ext}^1_\Lambda(H_m(X_{p,n};\Lambda),\Lambda)$ to
$H^{m+1}(X_{p,n};\Lambda)$, which is free. Thus homology is free with
the asserted ranks. Localization and change of coefficients give
the remaining group statements.
The ring statements follow by extension of scalars in
\eqref{eq:short}, which splits as a sequence of modules because its
rightmost term is free. Invariants commute with this change of
coefficients because $2$ is invertible.
\end{proof}

Let $A_{\mathrm{PL}}(Z)$ denote Sullivan's commutative differential
graded algebra of rational polynomial forms. For such algebras,
$\simeq$ denotes a zigzag of quasi-isomorphisms.

\begin{corollary}
Each $X_{p,n}$, $0\leq p\leq n$, is rationally formal:
\begin{equation*}
 A_{\mathrm{PL}}(X_{p,n})\simeq(H^*(X_{p,n}),0).
\end{equation*}
Every algebraic morphism between two such varieties is formal, and
every defined rational Massey product of length at least three
contains zero. The rational cohomology ring determines the rational
homotopy type.
\end{corollary}
\begin{proof}
Theorem~\ref{thm:purity} is precisely purity of weight $m$ in degree
$m$. Apply \cite[Thm.~8.1 and Prop.~8.9(1)]{CH20}.
The Massey product assertion is the standard consequence of formality;
the last assertion uses simple connectedness
\cite[Cor.~3.4, P.~2371]{Hu10} and the Sullivan correspondence
\cite[Cor.~1.26]{Hess}. The relevant ring is
\eqref{eq:fiberproduct} for $1\leq p<n$, the invariant ring of
$(\mathbb{P}^{n})^2$ for $p=0$, and $\mathbb{Q}$ for $p=n$.
\end{proof}

\section{Poincar\'e polynomials and stable series}\label{sec:series}
\subsection{The finite formula}
For a complex variety $Z$, let $W_\bullet$ and $F^\bullet$ denote the
weight and Hodge filtrations on compactly supported cohomology. Our
conventions are
\begin{equation*}
\begin{aligned}
 E(Z;u,v)&=\sum_{i,a,b}(-1)^i
 \dim_{\mathbb{C}}\operatorname{Gr}_F^a
 \operatorname{Gr}_{a+b}^W H_c^i(Z;\mathbb{C})u^av^b,\\
 \beta_j^{\mathrm{vir}}(Z)&=\sum_i(-1)^{i+j}
 \dim_{\mathbb{Q}}\operatorname{Gr}_j^W H_c^i(Z;\mathbb{Q})
 =(-1)^j[t^j]E(Z;t,t).
\end{aligned}
\end{equation*}
Here $\operatorname{Gr}$ denotes the corresponding associated graded
space. Thus $P_Z^{\mathrm{vir}}(t)=\sum_j(-1)^j
\beta_j^{\mathrm{vir}}(Z)t^j$, as in \cite[\S3.1]{CLN}.
The alternative convention $E(Z;-t,-t)$ agrees with ours for the
even polynomials considered here.

\begin{proof}[Proof of Theorem~\ref{thm:finite}]
For $p<n$, the virtual formula
$P_{X_{p,n}}^{\mathrm{vir}}(t)=F_{p,n}(t^2)$ is
\cite[\S4.5, Ex.~4.2]{Hu21}, expressed in Gaussian polynomial
notation. By projectivity and Theorem~\ref{thm:purity},
\begin{equation*}
 E(X_{p,n};u,v)=\sum_{s\geq0}b_{2s}(X_{p,n})(uv)^s.
\end{equation*}
Specializing $u=v=t$ identifies each coefficient with
$[q^s]F_{p,n}(q)$, proving all the assertions.
For $p=n$, $X_{n,n}$ is a point.
\end{proof}

Let $g_{r,N}(s)$ count partitions of $s$ into at most $r$ parts,
each at most $N-r$, with $g_{r,N}(s)=0$ for $s<0$.
Then $G_{r,N}(q)=\sum_{s\geq0}g_{r,N}(s)q^s$.
For a condition $\mathcal{E}$, we write $\mathbf{1}_{\mathcal{E}}$ for
its indicator. For $k\geq1$, $j\geq0$, define
\begin{equation*}
\begin{split}
 \mu_k(j)&=\#\{(a,b):0\leq a\leq b\leq k,\ a+b=j\}\\
 &=\max\{0,\min(k,\lfloor j/2\rfloor)-\max(0,j-k)+1\}.
\end{split}
\end{equation*}

\begin{corollary}
For $0\leq p<n$ and $s\geq0$,
\begin{equation*}
\begin{split}
 b_{2s}(X_{p,n})={}&\frac{1}{2}\left[
 \sum_{a=0}^{s}g_{k,N}(a)g_{k,N}(s-a)
 +\mathbf{1}_{2\mid s}g_{k,N}(s/2)\right]\\
 &+\sum_{j=0}^{s}g_{k+1,N}(s-j)
 \bigl(\mathbf{1}_{j\leq m_k}-\mu_k(j)\bigr).
\end{split}
\end{equation*}
The trace term is zero when $s$ is odd. The same expression equals
$\beta_{2s}^{\mathrm{vir}}(X_{p,n})$ and the degree-$2s$ Betti number
with coefficients in any field of characteristic different from $2$.
\end{corollary}
\begin{proof}
The coefficient of $q^j$ in
$(\Pi_k(q)^2+\Pi_k(q^2))/2$ is $\mu_k(j)$.
Extract coefficients in \eqref{eq:finite} and apply
Corollary~\ref{cor:awaytwo}.
\end{proof}

\subsection{Limits and nonrationality}
\begin{proof}[Proof of Theorem~\ref{thm:stable}]
The Gaussian product gives
\begin{equation}\label{eq:gausstail}
 G_{r,N}(q)=U_r(q)\prod_{j=1}^{r}(1-q^{N-r+j})
 \equiv U_r(q)\pmod{q^{N-r+1}}.
\end{equation}
Define
\begin{equation*}
 \Delta_k(q)=\Pi_{m_k}(q)-\frac{\Pi_k(q)^2+\Pi_k(q^2)}{2}.
\end{equation*}
Since $\Delta_k=O(q^2)$, \eqref{eq:gausstail} implies
$F_{p,n}\equiv F_p\pmod{q^{n-p+1}}$ and gives \eqref{eq:fixed}.
The limit is rational, and its coefficients are integers because
they stabilize from finite Betti numbers.
As $k\longrightarrow\infty$, one has
$U_k,U_{k+1}\longrightarrow U$ and
$\Pi_k,\Pi_{m_k}\longrightarrow(1-q)^{-1}$.
The correction tends to
\begin{equation*}
 \frac{1}{1-q}-\frac{1}{2}\left(\frac{1}{(1-q)^2}
 +\frac{1}{1-q^2}\right)
 =-\frac{q^2}{(1-q)(1-q^2)}.
\end{equation*}
This proves \eqref{eq:stable}.

The Euler product $U$ converges locally uniformly for $|q|<1$ and
defines a holomorphic function there. For $0<q<1$, let
$\varepsilon=1-q$ and
$\rho(q)=q^2/((1-q)(1-q^2))$.
For every positive integer $M$,
\begin{equation}\label{eq:growth}
 U(q)\geq\prod_{j=1}^{M}(1-q^j)^{-1}
 \geq\frac{1}{M!\varepsilon^M},
\end{equation}
since $1-q^j\leq j\varepsilon$.
Also $\rho(q)\leq\varepsilon^{-2}$, so
$\rho(q)/U(q)\leq6\varepsilon$ by taking $M=3$.
Consequently
\begin{equation}\label{eq:lower}
 F_\infty(q)\geq\frac14U(q)^2
 \geq\frac{1}{4(M!)^2\varepsilon^{2M}}
 \quad(0<\varepsilon<1/24).
\end{equation}
A rational function has at most a finite order pole at $1$, which
contradicts this inequality for arbitrary $M$.
If the formal series were rational, its germ would agree with the
holomorphic expression \eqref{eq:stable} on the unit disk, so the
same contradiction would apply. Substitution $q=t^2$ gives the
assertion for $F_\infty(t^2)$ as well.
\end{proof}

\subsection{Optimal stabilization and coefficient growth}
\begin{corollary}\label{cor:sharp}
Let $L_0=3$ and $L_p=p+4$ for $p\geq1$.
For every $n>p\geq0$,
\begin{equation*}
\begin{aligned}
 F_{p,n}(q)-F_p(q)&=-q^{n-p+1}+O(q^{n-p+2}),\\
 F_p(q)-F_\infty(q)&=-q^{L_p}+O(q^{L_p+1}).
\end{aligned}
\end{equation*}
The first discrepancy between $F_{p,n}$ and $F_\infty$ is in degree
$\min\{n-p+1,L_p\}$, with coefficient $-1$ when the two exponents
differ and $-2$ when they coincide.
For an integer $s\geq0$ and $p\geq1$, agreement in every degree
through $q^s$ is equivalent to
\begin{equation*}
 p\geq\max\{1,s-3\},\quad n\geq p+s,\quad n>p.
\end{equation*}
For $p=0$, it is equivalent to $s\leq2$ and $n\geq s$, with $n>0$.
All coefficients of $F_\infty$ are nonnegative integers.
\end{corollary}
\begin{proof}
From \eqref{eq:gausstail},
\begin{equation*}
 G_{k,N}-U_k=-q^{n-p+1}+O(q^{n-p+2}),\quad
 G_{k+1,N}-U_{k+1}=O(q^{n-p}).
\end{equation*}
The square term in \eqref{eq:finite} therefore contributes
$-q^{n-p+1}$, while the trace term changes only from degree
$2(n-p+1)$. Since $\Delta_k=O(q^2)$, the correction changes only
from degree $n-p+2$. This proves the first identity.

For $k=p+1\geq2$, let $\Pi_\infty=(1-q)^{-1}$ and
$\Delta_\infty=-q^2/((1-q)(1-q^2))$.
The first product of two distinct tail factors has degree $2k+3$,
so
\begin{equation*}
 U_k=U\bigl(1-q^{k+1}\Pi_\infty+O(q^{2k+3})\bigr).
\end{equation*}
Since $m_k+1\geq k+4$ and $2k+2\geq k+4$,
\begin{equation*}
 \Delta_k-\Delta_\infty=q^{k+1}\Pi_\infty^2+O(q^{k+4}).
\end{equation*}
In \eqref{eq:fixed}, the term
$(U_{k+1}-U)\Delta_\infty$ and the trace and squared tail errors
all have order at least $k+4$. Thus
\begin{equation*}
 F_p-F_\infty
 =q^{k+1}U\Pi_\infty(\Pi_\infty-U)+O(q^{k+4})
 =-q^{k+3}+O(q^{k+4}),
\end{equation*}
because $\Pi_\infty-U=-q^2+O(q^3)$.
For $p=0$,
$F_0=((1-q)(1-q^2))^{-1}$ and
$F_0-F_\infty=-q^3+O(q^4)$.
The leading errors have the same sign, so they cannot cancel.
Agreement of the coefficients in all degrees at most $s$ is
equivalent to the least exponent of a nonzero difference being
greater than $s$. This gives the stated stabilization conditions.
Each stable coefficient is a finite Betti number for sufficiently large parameters.
\end{proof}

\begin{proposition}
For fixed $p\geq0$, let $k=p+1$ and $a_{p,s}=[q^s]F_p(q)$.
The sequence $(a_{p,s})$ is eventually quasipolynomial, and
\begin{equation*}
 a_{p,s}=\frac{s^{2k-1}}{2(k!)^2(2k-1)!}+O_p(s^{2k-2})
 \quad(s\longrightarrow\infty).
\end{equation*}
Here quasipolynomial means polynomial on each residue class modulo
some positive integer, and the implicit constant may depend on $p$.
\end{proposition}
\begin{proof}
All poles of $F_p$ are roots of unity, so partial fractions give
eventual quasipolynomiality. On each residue class the polynomial
has rational coefficients by interpolation from rational values.
Since $\Delta_k(1)=0$, the pole orders at $1$ of the three terms
$U_k(q)^2/2$, $U_k(q^2)/2$, and $U_{k+1}(q)\Delta_k(q)$ are
$2k$, $k$, and at most $k$, respectively. Hence
\begin{equation*}
 \lim_{q\longrightarrow1}(1-q)^{2k}F_p(q)=\frac{1}{2(k!)^2}.
\end{equation*}
At a root of unity of order $d\geq2$, their pole orders are bounded
by $2\lfloor k/d\rfloor$, $k$, and $\lfloor(k+1)/d\rfloor$,
respectively, all at most $k$. Thus the unique pole of order $2k$
is at $1$. Its leading contribution is
$\binom{s+2k-1}{2k-1}/(2(k!)^2)$; all lower-order terms are
$O_p(s^{2k-2})$. This includes $k=1$.
\end{proof}

Let $\mathsf{p}(s)$ be the ordinary partition number, with
$\mathsf{p}(0)=1$ and $\mathsf{p}(s)=0$ for $s<0$. Denote the
coefficient $[q^s]F_\infty(q)$ by $B_s$.

\begin{corollary}
For $s\geq0$,
\begin{equation}\label{eq:stablecoeff}
 B_s=\frac{1}{2}\left[\sum_{a=0}^{s}\mathsf{p}(a)\mathsf{p}(s-a)
 +\mathbf{1}_{2\mid s}\mathsf{p}(s/2)\right]
 -\sum_{j=0}^{s-2}\left(\left\lfloor\frac j2\right\rfloor+1\right)
 \mathsf{p}(s-2-j).
\end{equation}
The trace term is zero for odd $s$.
The coefficientwise stable virtual Betti numbers are $B_s$ in degree
$2s$ and zero in odd degrees. The sequence $(B_s)$ satisfies no
nonzero constant coefficient linear recurrence eventually.
\end{corollary}
\begin{proof}
Use $U(q)=\sum_s\mathsf{p}(s)q^s$ and
$((1-q)(1-q^2))^{-1}=\sum_{j\geq0}(\lfloor j/2\rfloor+1)q^j$ in \eqref{eq:stable}.
An eventual recurrence would make $F_\infty$ rational, in contradiction to Theorem~\ref{thm:stable}.
\end{proof}

\section{Homotopy kernels and relative homology}\label{sec:homotopy}
Let $M_{p,n}=\min\{2p+1,2(n-p)\}$.
We use Hu's natural stability isomorphisms
\begin{equation}\label{eq:homotopystability}
 \pi_m(s_{p,n}):\pi_m(X_{p,n})\xrightarrow{\ \cong\ }\pi_m(D(2)),
 \quad 0\leq m\leq M_{p,n},
\end{equation}
as stated in \cite[Thm.~1.3, P.~2366; Propositions~3.3 and~3.5]{Hu10}.
The finite stages are simply connected
\cite[Cor.~3.4, P.~2371]{Hu10}, as is $D(2)$ by stabilization.
For a map $A\longrightarrow Z$, relative groups refer to its
mapping cylinder pair.
We use the rational relative Hurewicz theorem in its usual form:
for maps of simply connected CW complexes, the first nonzero rational
relative homotopy and homology groups occur in the same degree and
are naturally isomorphic. This is the rational form of
\cite[Thm.~4.32, P.~366]{Hatcher}, using rationalization
\cite[Thm.~1.5 and Def.~1.7]{Hess}.

\begin{lemma}\label{lem:stablecohomology}
The integral homology and cohomology of $D(2)$ are finitely generated
in each degree. For $s\geq0$,
\begin{equation*}
\begin{gathered}
 H^{2s+1}(D(2);\Lambda)=H_{2s+1}(D(2);\Lambda)=0,\\
 H^{2s}(D(2);\Lambda)\cong\Lambda^{B_s},\quad
 H_{2s}(D(2);\Lambda)\cong\Lambda^{B_s}.
\end{gathered}
\end{equation*}
In particular, $P_{D(2)}(t)=F_\infty(t^2)$.
\end{lemma}
\begin{proof}
The relative homotopy groups of $s_{p,n}$ vanish through degree
$M_{p,n}$ by \eqref{eq:homotopystability}.
For $M_{p,n}\geq2$, relative Hurewicz gives the same integral
homology vanishing. In degrees zero and one, this follows from
connectedness and simple connectedness, respectively.
Thus stabilization induces an integral homology isomorphism for
$m<M_{p,n}$, and the universal coefficient theorem gives
\begin{equation}\label{eq:cohomologystability}
 s_{p,n}^*:H^m(D(2);K)\xrightarrow{\ \cong\ }H^m(X_{p,n};K)
 \quad(m<M_{p,n})
\end{equation}
for any coefficient ring $K$.
Every fixed degree is covered by a sufficiently large finite stage,
which has finite CW type. This proves integral finite generation.
For degree $2s$, choose $p\geq s+1$ and $n\geq p+s+2$.
Corollaries~\ref{cor:awaytwo} and~\ref{cor:sharp}, together with
\eqref{eq:cohomologystability}, give the displayed groups and ranks.
The comparison is induced by the stabilization maps and follows
from their connectivity.
\end{proof}

\begin{proposition}\label{prop:firstkernel}
Let $A\xrightarrow{u}Z\xrightarrow{v}W$ be maps of simply connected
CW complexes with finite dimensional rational cohomology in each
degree. Suppose that $\theta=v\circ u$ is a rational homotopy
equivalence. Define
$\Delta b_m=b_m(Z)-b_m(A)$ and $K_m(v)=\ker v_{m,\mathbb{Q}}$.
Then $\Delta b_m\geq0$, $v_{m,\mathbb{Q}}$ is surjective, and
$v^*H^*(W)$ is a ring retract of $H^*(Z)$.
For $\nu\geq2$, the conditions
\begin{equation*}
 \Delta b_m=0\ (0\leq m<\nu),\quad
 K_m(v)=0\ (2\leq m<\nu)
\end{equation*}
are equivalent. Under either condition, relative projection and
Hurewicz give
\begin{equation}\label{eq:firstkernel}
 K_\nu(v)\cong H_\nu(Z,A;\mathbb{Q}),\quad
 \dim_{\mathbb{Q}}K_\nu(v)=\Delta b_\nu,
\end{equation}
and a perfect evaluation pairing
\begin{equation*}
 K_\nu(v)\times
 \ker\bigl(u^*:H^\nu(Z)\longrightarrow H^\nu(A)\bigr)
 \longrightarrow\mathbb{Q}.
\end{equation*}
\end{proposition}
\begin{proof}
The isomorphism $\theta^*=u^*v^*$ makes $u^*$ surjective and $v^*$
injective, with retraction $(\theta^*)^{-1}u^*$.
Consequently
\begin{equation*}
 0\longrightarrow H^m(Z,A)\longrightarrow H^m(Z)
 \xrightarrow{u^*}H^m(A)\longrightarrow0
\end{equation*}
is exact and $\dim H^m(Z,A)=\Delta b_m$.
On rational homotopy, $u_{m,\mathbb{Q}}$ is injective and
$v_{m,\mathbb{Q}}$ is surjective. The relative homotopy sequence
therefore identifies $\pi_m(Z,A)\otimes\mathbb{Q}$ with
$\operatorname{coker}u_{m,\mathbb{Q}}$.
Projection identifies this cokernel with $K_m(v)$: the coset of
$z$ has the representative
\begin{equation*}
 z-u_{m,\mathbb{Q}}\theta_{m,\mathbb{Q}}^{-1}v_{m,\mathbb{Q}}z,
\end{equation*}
which belongs to $K_m(v)$.
The rational relative Hurewicz theorem then gives the equivalence and
\eqref{eq:firstkernel}. Universal coefficient duality and naturality
of evaluation give the stated pairing.
\end{proof}

\begin{proof}[Proof of Theorem~\ref{thm:first}]
The series assertion is Lemma~\ref{lem:stablecohomology}.
The Lawson--Michelsohn formula
\cite[Thm.~4.4, P.~253]{LM} states, with suitable integral generators,
\begin{equation*}
\begin{gathered}
 \pi_{2j}(BU)=\mathbb{Z},\quad \pi_{2j-1}(BU)=0,\\
 \pi_{2j}(\vartheta):\mathbb{Z}\longrightarrow\mathbb{Z},
 \quad z\longmapsto(j-1)!z\quad(j\geq1).
\end{gathered}
\end{equation*}
Thus $\vartheta$ is a rational homotopy equivalence.
By \eqref{eq:stablecoeff}, the even Betti numbers of $D(2)$ through
degree eight are $(1,1,2,3,6)$, whereas those of $BU$ are
$(1,1,2,3,5)$. Both spaces have no odd rational cohomology.
Proposition~\ref{prop:firstkernel} applies to
$BU\xrightarrow{i}D(2)\xrightarrow{J}K_{\mathrm{ev}}$ and gives
\begin{equation*}
 H_m(D(2),BU;\mathbb{Q})=0\ (m<8),\quad
 H_8(D(2),BU;\mathbb{Q})\cong\ker J_{8,\mathbb{Q}}\cong\mathbb{Q}.
\end{equation*}
Rational surjectivity onto $\pi_8(K_{\mathrm{ev}})\otimes\mathbb{Q}
=\mathbb{Q}$ gives $\pi_8(D(2))\otimes\mathbb{Q}\cong\mathbb{Q}^{2}$.

Choose generators $\beta\in\pi_8(BU)$ and
$e\in\pi_8(K_{\mathrm{ev}})$, and let $\alpha=i_*\beta$. Then
$J_*\alpha=6e$.
Choose $\gamma\in\pi_8(D(2))$ whose rational image is linearly
independent of the rational image of $\alpha$, and write $J_*\gamma=ae$, $a\in\mathbb{Z}$.
Then
\begin{equation}\label{eq:eta}
 \eta=6\gamma-a\alpha
\end{equation}
has infinite order and belongs to $\ker J_*$.
Exactness of tensoring with $\mathbb{Q}$ gives rank one for the
integral kernel. Finally, \eqref{eq:homotopystability} transfers
these assertions to each stage with $p\geq4$, $n-p\geq4$.
For $(p,n)=(4,8)$, the range ends at $M_{4,8}=8$, which suffices.
\end{proof}

For a map of pairs $(Z,\varnothing)\longrightarrow(Z,A)$, denote
the induced relative to absolute cohomology homomorphism by
$\epsilon_{Z,A}$, with coefficient ring understood.
Let $\widetilde z_j=J^*u_j\in H^{2j}(D(2);\mathbb{Z})$ and let
$z_j$ be its rational image.

\begin{proposition}\label{prop:relative}
The restriction $i^*:H^*(D(2);\mathbb{Z})\longrightarrow
H^*(BU;\mathbb{Z})$ is split surjective as a homomorphism of graded
rings, with section $c_j\longmapsto\widetilde z_j$.
If $I_{\mathbb{Z}}=\ker i^*$, then $\epsilon_{D(2),BU}$ identifies
$H^*(D(2),BU;\mathbb{Z})$ with $I_{\mathbb{Z}}$ as a graded ring
without unit, and
\begin{equation}\label{eq:integralsplit}
 H^*(D(2);\mathbb{Z})
 =\mathbb{Z}[\widetilde z_1,\widetilde z_2,\ldots]\oplus I_{\mathbb{Z}}
\end{equation}
as graded abelian groups. The rational polynomial summand is
$A_{\mathrm{cl}}=\mathbb{Q}[z_1,z_2,\ldots]$.
For $s\geq0$,
\begin{equation}\label{eq:relativegroups}
\begin{gathered}
 H^{2s+1}(D(2),BU;\Lambda)=H_{2s+1}(D(2),BU;\Lambda)=0,\\
 H^{2s}(D(2),BU;\Lambda)\cong\Lambda^{B_s-\mathsf{p}(s)},\quad
 H_{2s}(D(2),BU;\Lambda)\cong\Lambda^{B_s-\mathsf{p}(s)}.
\end{gathered}
\end{equation}
The integral relative groups are finitely generated in every degree
and have no odd primary torsion. The relative Poincar\'e series is
$P_{(D(2),BU)}(t)=F_{\mathrm{rel}}(t^2)$, where
\begin{equation}\label{eq:relative}
 F_{\mathrm{rel}}(q)=F_\infty(q)-U(q)
 =\frac{U(q)^2+U(q^2)}2
 -\left(1+\frac{q^2}{(1-q)(1-q^2)}\right)U(q).
\end{equation}
It is nonrational over $\mathbb{C}$.
Over any field of characteristic different from $2$, relative odd
groups vanish and even groups have dimension $B_s-\mathsf{p}(s)$.
\end{proposition}
\begin{proof}
By \cite[Thm.~3.6, P.~250--252]{LM},
$i^*\widetilde z_j=\vartheta^*u_j=c_j$.
Since $H^*(BU;\mathbb{Z})=\mathbb{Z}[c_1,c_2,\ldots]$, these lifts
define a ring section $\tau$ of $i^*$.
The relative cohomology sequence becomes
\begin{equation*}
 0\longrightarrow H^m(D(2),BU;\mathbb{Z})
 \xrightarrow{\epsilon_{D(2),BU}}H^m(D(2);\mathbb{Z})
 \xrightarrow{i^*}H^m(BU;\mathbb{Z})\longrightarrow0.
\end{equation*}
The first map preserves products and has image $I_{\mathbb{Z}}$.
The section gives \eqref{eq:integralsplit} and algebraic independence
of the $\widetilde z_j$. Over $\mathbb{Q}$, $\vartheta^*$ is an
isomorphism, so the polynomial summand is $A_{\mathrm{cl}}$.
Only the integral classes $u_j$ are needed; no polynomial presentation
of $H^*(K_{\mathrm{ev}};\mathbb{Z})$ is asserted.

Lemma~\ref{lem:stablecohomology} and the homology sequence of the
pair give finite generation of all relative integral groups.
Over $\Lambda$, the section and that lemma make the relative even
cohomology free of rank $B_s-\mathsf{p}(s)$ and the odd cohomology
zero. The same universal coefficient argument as in
Corollary~\ref{cor:awaytwo} gives the relative homology in
\eqref{eq:relativegroups}. Localization and coefficient change yield
the torsion and field statements.
Taking ranks gives \eqref{eq:relative}.
By \eqref{eq:lower},
$F_{\mathrm{rel}}(q)\geq U(q)^2/8$ for $q$ sufficiently close to $1$,
since $U(q)\longrightarrow\infty$. Equation~\eqref{eq:growth}
therefore excludes rationality, even after substituting $q=t^2$.
\end{proof}

Let $I=\ker(i^*:H^*(D(2))\longrightarrow H^*(BU))$ and
$I^m=I\cap H^m(D(2))$.
For a connected graded algebra $E$, let
$Q(E)=E^{>0}/(E^{>0})^2$ be its indecomposable quotient.
For a based space $Z$, denote the integral Hurewicz homomorphism
$\pi_m(Z)\longrightarrow H_m(Z;\mathbb{Z})$ by
$\operatorname{hur}_m$.

\begin{corollary}
The space $I^8$ is one dimensional, and evaluation gives a perfect
pairing $\ker J_{8,\mathbb{Q}}\times I^8\longrightarrow\mathbb{Q}$.
For the element $\eta$ in \eqref{eq:eta}, there is a unique
$\zeta_8\in I^8$ such that
$\langle\zeta_8,\operatorname{hur}_8(\eta)\rangle=1$.
It is indecomposable, and $\dim_{\mathbb{Q}}Q(H^*(D(2)))^8=2$.
Its pullback to every finite stage in Theorem~\ref{thm:first} is
nonzero and indecomposable and detects the corresponding kernel
class.
\end{corollary}
\begin{proof}
The pairing is Proposition~\ref{prop:firstkernel} in degree eight.
The nonzero evaluation on $\eta\otimes1$ determines the normalization.
Below degree eight, the cohomology is the classical polynomial subring,
by equality of Betti numbers and injectivity of $J^*$.
The decomposable degree eight subspace is therefore spanned by
$z_1^4,z_1^2z_2,z_2^2,z_1z_3$.
Restriction to $BU$ proves that these classes are independent and
that their span meets $I^8$ trivially. Thus the indecomposable
quotient has dimension $6-4=2$, represented by $z_4,\zeta_8$.
At a finite stage, let $\eta_{p,n}$ denote the unique class satisfying
$\pi_8(s_{p,n})(\eta_{p,n})=\eta$.
Naturality gives
$\langle s_{p,n}^*\zeta_8,\operatorname{hur}_8(\eta_{p,n})\rangle=1$.
Products of positive degree classes evaluate trivially on a sphere,
which also proves the finite stage indecomposability.
\end{proof}

\begin{proof}[Proof of Theorem~\ref{thm:unbounded}]
Equation~\eqref{eq:lower} gives
\begin{equation}\label{eq:excessgrowth}
 \frac{F_\infty(q)}{U(q)}\geq\frac14U(q),\quad 0<1-q<1/24.
\end{equation}
The right side grows faster than every finite power of $(1-q)^{-1}$.
If $H$ were generated over $A_{\mathrm{cl}}$ by finitely many
elements, their homogeneous components would also generate it.
By odd vanishing, take generators in degrees $2d_1,\ldots,2d_r>0$.
The resulting polynomial algebra surjection gives
\begin{equation*}
 \frac{F_\infty(q)}{U(q)}\leq\prod_{j=1}^{r}(1-q^{d_j})^{-1}
 \quad(0<q<1),
\end{equation*}
contradicting \eqref{eq:excessgrowth}.
A finite homogeneous basis of $Q_{\mathrm{cl}}(H)$ would lift to
algebra generators over $A_{\mathrm{cl}}$, by induction on degree;
conversely such generators span that quotient. This proves the
equivalence. Finite dimensionality in each degree and equality with
the classical subring below degree eight give the asserted
unbounded even degrees.

Let $(\bigwedge\mathcal{V},d)$ be a Sullivan minimal model of $D(2)$.
By \cite[Thm.~1.24, Cor.~1.26 and Def.~1.27]{Hess},
$\mathcal{V}^m$ is finite dimensional and
$(\mathcal{V}^m)^\vee\cong\pi_m(D(2))\otimes\mathbb{Q}$.
Here $\bigwedge\mathcal{V}$ denotes the free graded commutative
algebra, polynomial in even degrees and exterior in odd degrees.
Let $\kappa_m=\dim_{\mathbb{Q}}\ker J_{m,\mathbb{Q}}$.
Rational surjectivity of $J_*$ implies
\begin{equation*}
 \dim\mathcal{V}^m=
 \begin{cases}1+\kappa_m,&m\text{ even},\\
 \kappa_m,&m\text{ odd}.
 \end{cases}
\end{equation*}
If only finitely many $\kappa_m$ were nonzero, the Hilbert series of
the underlying cochain algebra would be $U(t^2)\Phi(t)$, where
\begin{equation*}
 \Phi(t)=\prod_{\substack{m\geq2\\m\text{ even}}}(1-t^m)^{-\kappa_m}
 \prod_{\substack{m\geq3\\m\text{ odd}}}(1+t^m)^{\kappa_m}
\end{equation*}
is a finite product. The Hilbert series of a graded vector space
$E$ is $\sum_m(\dim E^m)t^m$.
Cohomology dimensions do not exceed cochain dimensions, so
$F_\infty(t^2)/U(t^2)\leq\Phi(t)$ for $0<t<1$.
The right side has at most a finite order pole at $1$, again a
contradiction. Thus $\kappa_m>0$ for infinitely many $m\geq8$.
Exactness of tensoring with $\mathbb{Q}$ gives infinite order
integral kernel elements, and \eqref{eq:homotopystability} transfers
them to the indicated finite stages.
\end{proof}

\begin{remark}
The degree eight statement concerns the first kernel after
rationalization, not possible torsion kernels in smaller degrees.
The torsion restrictions proved here concern homology and cohomology,
not homotopy groups. Formula~\eqref{eq:integralsplit} is an additive
splitting, not a product decomposition of rings or spaces; the free
module identifications in \eqref{eq:relativegroups} do not specify
bases. Proposition~\ref{prop:relative} answers the degree two case of
\cite[\S7, Ques.~IV, P.~263]{LM} after inverting $2$ only.
Possible $2$-primary torsion is not determined.
The first kernel criterion does not identify higher homotopy kernels
with relative Betti numbers, and the unboundedness argument requires
no formality assertion for $D(2)$.
\end{remark}

\section{Examples}\label{sec:examples}
\begin{example}
For $X=C_{1,2}(\mathbb{P}^{3})$, Theorem~\ref{thm:finite} gives,
with $q=t^2$,
\begin{equation*}
 P_X(t)=1+q+2q^2+2q^3+4q^4+3q^5+4q^6+2q^7+2q^8.
\end{equation*}
Thus $\chi(X)=21$. It also gives the Betti numbers over every
field of characteristic different from $2$.
\end{example}

\begin{example}
The hypersurface case
$C_{n-1,2}(\mathbb{P}^{n})\cong\mathbb{P}^{\binom{n+2}{2}-1}$
is classical \cite[Rem.~3.6, P.~2372]{Hu10} and gives
$F_{n-1,n}=\Pi_{m_n}$.
For the finite stage example $(p,n)=(4,8)$, the parameter space
coefficients through $q^4$ are
\begin{equation*}
\begin{array}{c|rrrrr}
 s&0&1&2&3&4\\\hline
 b_{2s}(R)&1&1&3&5&11\\
 b_{2s}(Y)&1&2&4&7&11\\
 b_{2s}(B)&1&2&5&9&16\\
 b_{2s}(X_{4,8})&1&1&2&3&6
\end{array}
\end{equation*}
The last row follows from \eqref{eq:short}.
For $n\geq9$, the analogous calculation appears in \cite[Lem.~2.5 and Prop.~2.6, P.~2368--2369]{Hu10};
the $n=8$ table is obtained from the finite formula here.
Theorem~\ref{thm:first} applies because $M_{4,8}=8$.
\end{example}

\begin{example}
For $p=0$, $\Delta_1=0$ and
$F_0(q)=((1-q)(1-q^2))^{-1}$, so that $a_{0,s}=\lfloor s/2\rfloor+1=s/2+O(1)$.
For $p=1$, $\Delta_2=-q^2+q^5$, and
\begin{equation*}
 F_1(q)=\frac{1}{2}\bigl(U_2(q)^2+U_2(q^2)\bigr)-q^2U_2(q).
\end{equation*}
The first eleven coefficients are
$(1,1,2,3,6,8,13,17,25,31,42)$.
They agree with those of $F_{1,n}$ for $n\geq11$.
\end{example}

\begin{example}
Formula~\eqref{eq:stablecoeff} gives
\begin{equation*}
\begin{array}{c|rrrrrrrrrrr}
 s&0&1&2&3&4&5&6&7&8&9&10\\\hline
 B_s&1&1&2&3&6&9&17&27&48&77&130\\
 B_s-\mathsf{p}(s)&0&0&0&0&1&2&6&12&26&47&88
\end{array}
\end{equation*}
By Corollary~\ref{cor:sharp}, the first row already gives the even Betti numbers through degree twenty of $C_{7,2}(\mathbb{P}^{17})$.
This is coefficient agreement, not homotopy stability in that range.
If $(p,n)=(1,5)$, the leading exponents of the two stabilization
differences coincide:
\begin{equation*}
 F_{1,5}(q)=1+q+2q^2+3q^3+6q^4+7q^5+O(q^6),
\end{equation*}
whose first discrepancy from $F_\infty$ is $-2q^5$.
\end{example}

\end{document}